\documentclass[reqno]{amsart}
\usepackage{amsmath,amsthm,amsfonts,hyperref}
\usepackage{url}
\usepackage{mathrsfs}

\newcommand{\nexteq}{\displaybreak[0]\\ &=}

\usepackage[shortlabels]{enumitem}
\setlist[enumerate,1]{label={\upshape(\roman*)}}

\newtheorem{thm}{Theorem}
\newtheorem{lem}[thm]{Lemma}

\theoremstyle{definition}

\newtheorem{remark}[thm]{Remark}

\DeclareMathOperator{\rank}{rank}
\DeclareMathOperator{\tr}{Tr}
\newcommand{\C}{\mathbb{C}}
\newcommand{\R}{\mathbb{R}}

\newcommand{\Z}{\mathbb{Z}}

\title{Bordered complex Hadamard matrices and strongly regular graphs}
\author{Takuya Ikuta}
\address{Faculty of Law,
Kobe Gakuin University, Kobe, 650-8586, Japan}
\email{ikuta@law.kobegakuin.ac.jp}
\author{Akihiro Munemasa}
\address{Graduate School of Information Sciences,
Tohoku University, Sendai, 980-8579, Japan}
\email{munemasa@math.is.tohoku.ac.jp}
\thanks{The work of the authors was supported by JSPS KAKENHI grant number JP20K03527.}


\begin{document}

\begin{abstract}
We consider bordered complex Hadamard matrices 
whose core is contained in the Bose--Mesner algebra of a strongly regular graph. 
Examples include a complex Hadamard matrix 
whose core is contained in the Bose--Mesner algebra of
a conference graph due to J.~Wallis, F.~Sz\"{o}ll\H{o}si,
and a family of Hadamard matrices
given by S.~N.~Singh and O.~P.~Dubey.
In this paper, we 
prove that there are no other
bordered complex Hadamard matrices 
whose core is contained in the Bose--Mesner algebra of a strongly regular graph. 
\end{abstract}

\maketitle

\section{Introduction}
A complex Hadamard matrix is a square matrix 
$W$ of order $n$ which satisfies $W\overline{W}^\top= nI$ and
all of whose entries are complex numbers of absolute value $1$.
They are the natural generalization of real Hadamard matrices. 
Complex Hadamard matrices appear frequently in various branches of mathematics and 
quantum physics.

We consider the following ``bordered'' matrix of the form:
\begin{equation}\label{eq:W}
W=\begin{pmatrix}
1 & \boldsymbol{e} \\
\boldsymbol{e}^\top & W_1
\end{pmatrix},
\end{equation}
where $\boldsymbol{e}$ is the all $1$'s row vector of size $n$.
The submatrix $W_1$ is said to be the core of $W$.
In this paper, we consider complex Hadamard matrices $W$
of the form \eqref{eq:W} whose core $W_1$ is contained
in the Bose--Mesner algebra of a symmetric $2$-class association scheme.
In \cite{Wallis,FS} J.~Wallis and F.~Sz\"{o}ll\H{o}si constructed a complex Hadamard matrix $W$ 
whose core $W_1$ is contained in the Bose--Mesner algebra of a conference graph.
And, in \cite{Singh-Dubey} S.~N.~Singh and O.~P.~Dubey
constructed a Hadamard matrix $W$ whore core $W_1$ 
is contained in the Bose--Mesner algebra of strongly regular graph with 
$(k,\lambda,\mu)=(2r^2,r^2,r^2)$.
As a natural problem,
assuming $W_1$ is contained in the Bose--Mesner algebra of  a strongly regular graph,
we are interested in whether $W$ is a complex Hadamard matrix or not.

A similar problem has been considered in our earlier papers
(see \cite{MI,MI2,FS} and references therein).
In \cite{MI,MI2}, we considered borderless complex Hadamard matrices 
contained in the Bose--Mesner algebra 
of some association schemes.

Let $X$ be a finite set with $n$ elements, and let
$\mathfrak{X}=(X,\{R_i\}_{i=0}^2)$ be a symmetric $2$-class association scheme 
with the first eigenmatrix $P=(P_{i,j})_{0\leq i, j\leq 2}$:
\begin{equation}\label{eq:d2P}
\begin{pmatrix}
1&k&\ell\\
1&r&-(r+1)\\
1&s&-(s+1)
\end{pmatrix},
\end{equation}
where $r,s\in\R$, $r\geq 0$, and $s\leq-1$.
We let $\mathfrak{A}$ denote the Bose--Mesner algebra
spanned by the adjacency matrices $A_0,A_1,A_2$ of $\mathfrak{X}$.
A strongly regular graph $\Gamma$ with parameters $(k,\lambda,\mu)$ 
is equivalent to $\mathfrak{X}$, via the correspondence 
$R_1$ equal to the set of edges and $R_2$ equal to the set of non-edges.
In this paper, 
by exchanging $R_1$ and $R_2$, we may assume that $r+s\geq-1$
without loss of generality.

Let 
\begin{equation} \label{eq:W1}
W_1=w_0A_0+w_1A_1+w_2A_2\in\mathfrak{A}.
\end{equation}
Suppose that $w_0,w_1,w_2$ are complex numbers of absolute value $1$,
and $w_1\not=w_2$.
Then we have the following.

\begin{thm} \label{thm:main}
Suppose that $r,s\in\R$, $r\geq0$, $s\leq-1$, $r+s\geq-1$, and $w_1\not=w_2$.
Let $W_1$ be the matrix defined in {\rm(\ref{eq:W1})}.
If the matrix $W$ defined by {\rm(\ref{eq:W})} is a complex Hadamard matrix,
then one of the following holds.
\begin{enumerate}
\item\label{ti}
$\Gamma$ has parameter $(k,\lambda,\mu)=(2r^2,r^2,r^2)$, and
$(w_0,w_1,w_2)=(1,-1,1)$.
\item\label{tii}
$\Gamma$ is a conference graph on $2k+1$ vertices, and
\begin{itemize}
\item[{\rm(a)}] $(w_0,w_1,w_2)=(-1,\pm i,\mp i)$, and 
\item[{\rm(b)}] $(w_0,w_1,w_2)=(1,
\frac{-1\pm i\sqrt{k^2-1}}{k},
\frac{-1\mp i\sqrt{k^2-1}}{k})$.
\end{itemize}
\end{enumerate}
Conversely, if  {\rm(i)} or {\rm(ii)} hold, then 
$W$ is a complex Hadamard matrix.
\end{thm}

\begin{remark}\label{rem:2}
Strongly regular graphs having parameters \ref{ti} in Theorem~\ref{thm:main} was considered in \cite{Singh-Dubey}.
The list of strongly regular graphs up to $1,300$ vertices
are given in Brouwer's database \cite{Brouwer}.
According to that,
strongly regular graphs with such a parameter
exist for $r=2,\ldots,10,12,\ldots,16,18$, and 
are unknown for $r=11,17$.

Complex Hadamard matrices having (a) and (b) in Theorem~\ref{thm:main}~\ref{tii}
were considered in \cite{Wallis} and 
\cite[Proposition~3.4.16]{FS}, respectively.
The matrix in (b) of Theorem~\ref{thm:main}~\ref{tii}
is a Butson-type complex Hadamard matrix 
if and only if $k=2$
\cite[Remark~3.4.17]{FS}.
If a conference graph on
$2k+1$ vertices exists, then $k$ must be even. 
A conference graph on $2k+1$ vertices is known to exist
for $k=2,4,\dots,30$ except $k=10,16,28$, for which the 
nonexistence is known. The existence is undecided for
$k=32$.
\end{remark}

\begin{remark}\label{rem:3}
Two complex Hadamard matrices $W$ and $W'$ are said to be
equivalent if
there exist diagonal matrices $D, D'$ with nonzero complex diagonal entries, and permutation matrices $T, T'$, such that
$DWD'=TW'T'$ holds.
In Theorem~\ref{thm:main} (i), both parts (a) and (b)
give two conjugate complex Hadamard matrices of order $n=2k+1$.
The matrices in (a) and (b) are never equivalent. This can be
seen by computing the Haagerup set
$H(W)$ (see \cite{Haagerup}) defined as
\[
H(W)= \left\{\frac{W_{i_1,j_1}W_{i_2,j_2}}{W_{i_1,j_2}W_{i_2,j_1}}
\Bigm| 1\leq i_1,i_2,j_1,j_2\leq n\right\},
\]
which is an invariant for equivalence.
Indeed, for a matrix $W$ in part (a) of
Theorem~\ref{thm:main}~\ref{tii}, 
we have $H(W)\subseteq\{\pm1,\pm i\}$,
while for a matrix $W$ in part (b) of
Theorem~\ref{thm:main}~\ref{tii}, we have $H(W)\cap\{\pm1,\pm i\}=\{1\}$.

As for the two conjugate matrices in part (a) or in part (b),
they are equivalent if the corresponding conference graph
is self-complementary. Otherwise, it is unclear whether
the two conjugate matrices are equivalent or not.
There are non-self-complementary conference graphs of 
order $25$, according to \cite{H}.
\end{remark}

The organization of the paper is as follows.
After giving preliminaries in Section~\ref{2-class}, 
we give an overview on strongly regular graphs in Section~\ref{sec.srg}.
We also prove the ``converse'' part of
Theorem~\ref{thm:main} in Section~\ref{sec.srg}.
It then remains to derive \ref{ti} and \ref{tii} of
Theorem~\ref{thm:main} under the hypotheses of that theorem.
In Section~\ref{valueOfw1}, we give a quadratic equation satisfied by
the real part of $w_1$ (see {\eqref{eq:W1}}), and a necessary 
condition that the real part lies in the interval $[-1,1]$.
In Section~\ref{sec.rs-1}, we consider the special case $r+s=-1$, and
derive Theorem~\ref{thm:main}~(ii). In Section~\ref{sec.6}, we take a closer look
at the properties of the polynomials $L(X),M(X)$ and $S(X)$
which are needed to express the real part of $w_1$.
As a result, we obtain Theorem~\ref{thm:main}~\ref{ti} under the assumption
$r+s=0$. Finally, in Section~\ref{sec:rs-p}, we 
rule out the case $r+s>0$.

All the computer calculations in this paper were performed 
with the help of Magma \cite{magma}.

\section{Preliminaries}\label{2-class}
First we consider a more general situation than the one mentioned
in the Introduction.
Let $(X,\{R_i\}_{i=0}^d)$ be a symmetric $d$-class association scheme 
with the first eigenmatrix $P=(P_{i,j})_{0\leq i, j\leq d}$.
For more general and detailed theory of association schemes, see \cite{BI}.
We let $\mathfrak{A}$ denote the Bose--Mesner algebra
spanned by the adjacency matrices $A_0,A_1,\ldots,A_d$ of $\mathfrak{X}$.
Then the adjacency matrices are expressed as
\begin{equation}\label{eq:pij}
A_j=\sum_{i=0}^d P_{i,j}E_i \quad (j=0,1,\ldots,d),
\end{equation}
where $E_0=\frac{1}{n}J,E_1,\ldots,E_d$ are the primitive idempotents of $\mathfrak{A}$.

Let
\begin{equation}\label{eq:WW1}
W_1=\sum_{j=0}^d w_jA_j\in\mathfrak{A},
\end{equation}
where $w_0,\ldots,w_d$ are complex numbers of absolute value $1$.
Define
\begin{equation}\label{eq:beta}
\beta_i=\sum_{j=0}^d w_jP_{i,j}  \quad (i=0,1,\ldots,d).
\end{equation}
By \eqref{eq:pij}, \eqref{eq:WW1} and \eqref{eq:beta} we have
\begin{equation}\label{eq:WE}
W_1=\sum_{i=0}^d \beta_iE_i.
\end{equation}
Let $X_i$ {\rm($0\leq i\leq d$)} be indeterminates. For $j=1,2,\dots,d$, 
let $e_j$ be the polynomial defined by
\begin{equation}\label{eq:ek}
e_j=\prod_{i=0}^dX_h
 \left(\sum_{i=0}^dP_{j,i}^2+\sum_{0\leq j_1<j_2\leq d}P_{j,j_1}P_{j,j_2}\left(\frac{X_{j_1}}{X_{j_2}}+\frac{X_{j_2}}{X_{j_1}}\right)-(n+1)\right),
\end{equation}
and $e_0$ be the polynomial defined by
\begin{equation}\label{eq:e0}
e_0=1+\sum_{j=0}^dP_{0,j}X_j.
\end{equation}
Then we have the following.

\begin{lem}\label{lem:equiv}
The following statements are equivalent:
\begin{enumerate}
\item\label{equiv1} The matrix $W$ defined by \eqref{eq:W} is a complex Hadamard matrix,
\item\label{equiv2} $\beta_i\overline{\beta_i}=n+1$ for $i=1,\ldots,d$, and $1+\sum_{j=0}^d P_{0,j}w_j=0$,
\item\label{equiv3} $(w_i)_{0\leq i\leq d}$ is a common zero of $e_j$ {\rm($j=0,\ldots,d$)}.
\end{enumerate}
\end{lem}
\begin{proof}
By (\ref{eq:W}) we have
\begin{align*}
W\overline{W}^\top
=&
\begin{pmatrix}
n+1 & \boldsymbol{e}(I+\overline{W_1}^\top) \\
(I+W_1)\boldsymbol{e}^\top & J+W_1\overline{W_1}^\top
\end{pmatrix}. \label{WWt}
\end{align*}
By (\ref{eq:WE}) we have
\begin{equation}\label{WWt2}
W_1\overline{W_1}^\top=\sum_{i=0}^d\beta_i\overline{\beta_i}E_i.
\end{equation}

Suppose that the matrix (\ref{eq:W}) is a complex Hadamard matrix.
Since $W\overline{W}^\top=(n+1)I$,
we have
\begin{align}
  W_1\overline{W_1}^\top&=(n+1)I-J\nonumber\\
  &=E_0+(n+1)\sum_{j=1}^dE_j,\label{eq:3}\\
  (I+W_1)\boldsymbol{e}^\top&=0.\label{eq:4}
\end{align}
Therefore, by (\ref{WWt2}), (\ref{eq:3}), and (\ref{eq:4}),
\ref{equiv1} implies \ref{equiv2}.

To prove the converse, it suffices to show $\beta_0\overline{\beta_0}=1$.
Since $W$ is symmetric, the diagonal entries of $W_1\overline{W_1}^\top$ are
all $n$. Thus
\begin{align*}
n^2&=
\tr W_1\overline{W_1}^\top
\nexteq
\sum_{j=0}^d\beta_j\overline{\beta_j}\tr E_j
&&\text{(by \eqref{WWt2})}
\nexteq
\beta_0\overline{\beta_0}+\sum_{j=1}^d (n+1)\tr E_j
\nexteq
\beta_0\overline{\beta_0}+(n+1)\tr (I-E_0)
\nexteq
\beta_0\overline{\beta_0}+(n+1)(n-1),
\end{align*}
and hence $\beta_0\overline{\beta_0}=1$.

By (\ref{eq:beta}) we have
\[
\beta_i\overline{\beta_i}=
\sum_{j=0}^dP_{i,j}^2
+\sum_{0\leq j_1<j_2\leq d}P_{i,j_1}P_{i,j_2}\left(\frac{w_{j_1}}{w_{j_2}}+\frac{w_{j_2}}{w_{j_1}}\right)
 \]
for $i=1,\dots, d$.
Therefore, the equivalence of \ref{equiv2} and \ref{equiv3} follows.
\end{proof}

The following is analogous to \cite[Proposition 2.2]{Chan}.
\begin{lem}\label{lem:0324-1}
If the matrix $W$ defined by {\rm(\ref{eq:W})} is a complex Hadamard matrtix,
then we have
\[
n+1\leq\left(\sum_{i=0}^d|P_{j,i}|\right)^2 \quad (j=0,1,\ldots, d).
\]
\end{lem}
\begin{proof}
By \ref{equiv2} in Lemma~\ref{lem:equiv}, we have
\begin{align*}
n+1&=\beta_j\overline{\beta_j} \\
&=\left(\sum_{j_1=0}^d w_{j_1}P_{j,j_1}\right)\left(\sum_{j_2=0}^d \frac{P_{j,j_2}}{w_{j_2}}\right) \\
&=\sum_{i=0}^dP_{j,i}^2+\sum_{0\leq j_1<j_2\leq d}\left(\frac{w_{j_1}}{w_{j_2}}+\frac{w_{j_2}}{w_{j_1}}\right)P_{j,j_1}P_{j,j_2}.
\end{align*}
for $j=0,1,\ldots,d$.
Since $W$ is a complex Hadamard matrix, we have
$\left|\frac{w_{j_1}}{w_{j_2}}+\frac{w_{j_2}}{w_{j_1}}\right|\leq2$.
Then
\begin{align*}
n+1&\leq\sum_{i=0}^d|P_{j,i}|^2+
\sum_{0\leq j_1<j_2\leq d}\left|\frac{w_{j_1}}{w_{j_2}}+\frac{w_{j_2}}{w_{j_1}}\right||P_{j,j_1}||P_{j,j_2}| \\
&\leq \sum_{i=0}^d|P_{j,i}|^2+2\sum_{0\leq j_1<j_2\leq d}|P_{j,j_1}||P_{j,j_2}| \\
&=\left(\sum_{i=0}^d|P_{j,i}|\right)^2.&\qedhere
\end{align*}
\end{proof}

Let
$f(X)$ be a non-constant polynomial with real coefficients.
Put $f_0(X)=f(X)$ and $f_1(X)=f_0'(X)$.
Define 
\[f_{j+1}(X)=-\text{Rem}(f_{j-1}(X),f_{j}(X))\quad (j=1,2,\ldots)\]
where, for polynomials $a(X)$, $b(X)\ne0$, we denote 
by $\text{Rem}(a(X),b(X))$ the remainder when $a(X)$ is reduced modulo $b(X)$. There exists a positive integer $m$ such that
$f_m(X)\not=0$ and $f_{m+1}(X)=0$.
The sequence of the polynomials
\[
f_0(X), f_1(X), f_2(X),\ldots,f_m(X)
\]
is called the Strum sequence associated to $f(X)$.

Let $c_j$ be the leading coefficient of $f_j(X)$, and $d_j=\deg f_j(X)$ for $j=0,1,\ldots,m$.
Then we have the following sequences:
\begin{align}
& \left(\text{sgn}(c_j)\right)_{j=0}^m, \label{seq:1} \\
& \left(\text{sgn}((-1)^{d_j}c_j)\right)_{j=0}^m. \label{seq:2}
\end{align}

\begin{thm}[Sturm \cite{Sturm}; see also {\cite[Corollary 10.5.4]{Rahman}}]\label{thm:sturm}
With the above notation, 
the number of distinct real roots of $f(X)$ is given by 
the number of sign changes of \eqref{seq:2} minus
the number of sign changes of \eqref{seq:1}.
\end{thm}

\section{Strongly regular graphs}\label{sec.srg}
In this section, we review basic properties of
symmetric $2$-class association schemes and 
strongly regular graphs.
Let $\mathfrak{X}=(X,\{R_i\}_{i=0}^2)$ be a symmetric $2$-class association scheme
with the first eigenmatrix (\ref{eq:d2P}).
We have the following three cases in (\ref{eq:d2P}): (i) $r+s\geq0$, (ii) $r+s=-1$, (iii) $r+s\leq-2$.
Suppose that (iii) holds.
Then the eigenvalues of $R_2$ satisfy $-(r+1)-(s+1)\geq0$.
By exchanging $R_1$ and $R_2$, 
we may assume that $r+s\geq-1$ without loss of generality.
Therefore we only consider the two cases (i) and (ii).
Under this assumption, we have
\begin{equation}\label{k2}
\ell\geq2.
\end{equation}
Indeed, if $\ell=1$, then $R_2$ is a matching, and hence the eigenvalues
satisfy $-r-1=-1$ and $-s-1=1$. This implies $r+s=-2$, contrary to our
assumption.

A strongly regular graph $\Gamma$ with parameters $(k,\lambda,\mu)$ 
is equivalent to $\mathfrak{X}$, via the correspondence 
$R_1$ equal to the set of edges and $R_2$ equal to the set of non-edges.
The complement of a strongly regular graph is also a strongly regular graph.
Then we have
\begin{align}
\mu&=k+rs, \label{mu}\\
\lambda&=r+s+\mu, \label{lambda}\\
\ell\mu&=k(k-\lambda-1),\label{eqk2}\\
n\mu&=(1+k)\mu+\ell(k-\lambda-1). \nonumber
\end{align}
Let $m_j=\rank E_j$ for $j=1,2$.
Then we have
\begin{align}
m_1&=\frac12(n-1-\frac{2k+(n-1)(\lambda-\mu)}{\sqrt{(\lambda-\mu)^2+4(k-\mu)}}), \label{m1} \\
m_2&=\frac12(n-1+\frac{2k+(n-1)(\lambda-\mu)}{\sqrt{(\lambda-\mu)^2+4(k-\mu)}}). \nonumber
\end{align}

A conference graph is a strongly regular graph $\Gamma$
satisfying one of the following two equivalent conditions:
\begin{itemize}
\item[(i)] $k=2r(r+1)$, $r+s=-1$,
\item[(ii)] $m_1=m_2$.
\end{itemize}
We remark that the eigenvalues $r,s$ of a strongly regular graph $\Gamma$
are integers unless
$\Gamma$ is a conference graph.
If $\Gamma$ is a conference graph, then $r=\frac{-1+\sqrt{2k+1}}{2}$ and $s=\frac{-1-\sqrt{2k+1}}{2}$. In any case, 
\begin{equation}\label{rsZ}
rs\in\Z.
\end{equation}
Note that we allow disconnected strongly regular graphs.
These graphs are characterized by $s=-1$, or equivalently,
$\mu=0$. 

By \eqref{eq:d2P}, \eqref{eq:ek}, and \eqref{eq:e0} we have
\begin{align}
e_0=&1+X_0+kX_1+\ell X_2, \label{eq0}\\
e_1=&-\left((r+1)X_1-rX_2\right)X_0^2
        -\left( r(r+1)(X_1-X_2)^2+(k+\ell)X_1X_2\right)X_0 \nonumber \\
        &+(rX_1-(r+1)X_2)X_1X_2,  \label{eq1} \\
e_2=&-\left((s+1)X_1-sX_2\right)X_0^2
        -\left(s(s+1)(X_1-X_2)^2+(k+\ell)X_1X_2\right)X_0 \nonumber \\
        &+(sX_1-(s+1)X_2)X_1X_2.  \label{eq3}
\end{align}
Let $\mathcal{I}$ be the ideal of the polynomial ring $\mathscr{R}=\mathbb{C}[X_0,X_1,X_2]$ generated by
(\ref{eq0}), (\ref{eq1}), and (\ref{eq3}).

\begin{lem}\label{lem:0918}
Let $W_1$ be the matrix defined by {\rm(\ref{eq:W1})},
and let $W$ be the matrix defined by {\rm(\ref{eq:W})}.
Then $W$ is a complex Hadamard matrix if and only if
$(w_0,w_1,w_2)$ is a common zero of the polynomials $e_k$ {\rm($k=0,1,2$)}.
\end{lem}
\begin{proof}
This follows easily from Lemma~\ref{lem:equiv}
by setting $d=2$.
\end{proof}

\begin{proof}[Proof of the ``converse'' part of 
Theorem~\ref{thm:main}]
Assume that the matrix $W$ is one of the matrices 
\ref{ti}, \ref{tii} in Theorem~\ref{thm:main}.
In view of Lemma~\ref{lem:0918},
it suffices to show that
$(w_0,w_1,w_2)$ is a common zero of
the polynomials \eqref{eq0}, \eqref{eq1}, and \eqref{eq3}.
This can be done by direct calculation.
\end{proof}

\begin{lem}\label{lem:rEq0}
Let $W_1$ be the matrix defined by \eqref{eq:W1},
and let $W$ be the matrix defined by \eqref{eq:W}.
If $W$ is a complex Hadamard matrix, then we have the following:
\begin{enumerate}
\item\label{rEq0_1} $s<-1$,
\item\label{rEq0_2} $n+1\leq4s^2$.
\end{enumerate}
\end{lem}
\begin{proof}
\ref{rEq0_1} Suppose that $s=-1$.
Then, since the graph $(X,R_1)$ is the union of complete graphs,
we have $k=r$.
We can verify that $\mathcal{I}$ contains $X_2(\ell X_2+1)^2(X_2^2+(k+\ell)X_2+1)$.
By Lemma~\ref{lem:0918},
$(w_0,w_1,w_2)$ is a common zero of
the polynomials (\ref{eq0}), (\ref{eq1}), and (\ref{eq3}) in $\mathcal{I}$.
From this, $\ell w_2+1=0$ or $w_2^2+(k+\ell)w_2+1=0$.
By \eqref{k2}, we have $\ell w_2+1\ne0$.
Since $|w_2^2+1|<3\leq k+\ell=|(k+\ell)w_2|$, 
we have $w_2^2+(k+\ell)w_2+1\ne0$.
This is a contradiction.

\ref{rEq0_2} Applying Lemma~\ref{lem:0324-1} for $j=2$, we have
\begin{align*}
n+1&\leq \left(\sum_{i=0}^2|P_{2,i}|\right)^2 \\
&=(1+(-s)-(s+1))^2 \\
&=4s^2.&\qedhere
\end{align*}
\end{proof}

If $s\neq-1$, then
we have $\mu>0$ by \eqref{mu}, \eqref{lambda}, and \eqref{eqk2}.
Thus
\begin{equation}\label{mu0}
\ell=\frac{-k(r+1)(s+1)}{k+rs}.
\end{equation}

\begin{remark}\label{rem:9}
Applying Lemma~\ref{lem:0324-1} for $j=1$, we have 
$n+1\leq4(r+1)^2$. This inequality is weaker than the one
stated in \ref{rEq0_2} of Lemma~\ref{lem:rEq0}. Indeed, since
we assumed that $r+s\geq-1$ in the beginning of this section, 
we have $(r+1)^2\geq s^2$.
\end{remark}

\section{The real part of $w_1$}\label{valueOfw1}
We suppose that $r,s\in\R$, $r\geq0$, $s<-1$, and $r+s\geq-1$.
Let $W_1$  be the matrix defined by (\ref{eq:W1}), and
$W$ be the matrix defined by (\ref{eq:W}).
We suppose that 
the matrix $W$ is a complex Hadamard matrix.
Let
\[
w_j=a_j+b_ji
\]
for $j=0,1,2$, where $a_j,b_j\in\R$, $a_j^2+b_j^2=1$, and $i^2=-1$.
Recall that $\mathcal{I}$ is the ideal of $\mathscr{R}$ 
generated by
(\ref{eq0}), (\ref{eq1}), and (\ref{eq3}),
and $(w_0,w_1,w_2)$ is a common zero of $\mathcal{I}$
by Lemma~\ref{lem:0918}. Since we assume $w_1\neq w_2\ne0$,
in Theorem~\ref{thm:main}, we consider the ideal
$\tilde{\mathcal{I}}$ of the polynomial ring 
$\tilde{\mathscr{R}}=\C[X_\infty,X_0,X_1,X_2]$ 
generated by
(\ref{eq0}), (\ref{eq1}), (\ref{eq3}),
and 
\[e_\infty=1+X_\infty(r-s)(X_1-X_2)X_2.\]
Note that we have included the factor $r-s$ for a technical reason.
In computer implementation, we regard $r$ and $s$ as indeterminates as well, 
but $r$ and $s$ are assumed to take distinct values.

Define the polynomials $L(X)$, $M(X)$, and $S(X)$ as follows:
\begin{align}
L(X)&=X^3+\frac{4rs-r-s+3}{2}X^2+\frac{-4rs(r+s-1)+1}{2}X \nonumber\\
&\quad +\frac{rs(r^2+2(3s+1)r+s^2+2s+2)}{2}, \label{eq:1224-1} \displaybreak[0]\\
M(X)&=L(X)-4(X+rs)^2, \label{eq:1224-2} \displaybreak[0]\\
S(X)&=s_4X^4+s_3X^3+s_2X^2+s_1X+s_0, \label{eq:1224-3}
\end{align}
where
\begin{align*}
s_4&=(r+s+1)^2,\\
s_3&=4sr^3+8s(s+1)r^2+(4s^3+8s^2+8s+2)r+2s+2,\\
s_2&=2s(2s-1)r^4+2s(s+1)(4s-3)r^3+2s(2s^3+s^2+6s+4)r^2 \\
 &\quad -2s(s+1)(s^2+2s-6)r+1,\\
s_1&=-2rs(2sr^4+6s(s+1)r^3+(6s^3-4s^2-8s-1)r^2 \\
 &\quad +2(s+1)(s^3+2s^2-6s-1)r-s^2-2s-2),\\
s_0&=r^2s^2(r^4+4(s+1)r^3+(22s^2+28s+8)r^2 \\
 &\quad +4(s+1)(s^2+6s+2)r+(s^2+2s+2)^2).
\end{align*}

\begin{lem}\label{lem:a1}
We have
\begin{align}
(L(k)-M(k))^2a_1^2+2(L(k)^2-M(k)^2)a_1+(L(k)+M(k))^2-S(k)&=0, \label{eq:0115}\\
2(k+rs)^2rsa_0-2k^2(k+rs)^2a_1+h_0&=0, \label{eq:0804} \\
2(k+rs)^3a_1-2(k+rs)r(r+1)s(s+1)a_2+\ell_0&=0, \label{eq:0805}
\end{align}
where
\begin{align*}
h_0&=-k^5+(r+s-2rs+1)k^4+3rs(r+s+1)k^3 \\
 &\quad -rs((r+s)^2+2(r+s)-1)k^2+4r^2s^2k+2r^3s^3,\\
\ell_0&=k(k-r-s-1)(k^2+2rsk-rs(r+s+1)).
\end{align*}
\end{lem}
\begin{proof}
With the help of Magma,
we can verify that $\tilde{\mathcal{I}}$ contains
$f_1(X_0,X_1,X_2)$ and 
$f_2(X_0,X_1,X_2)$,
where
\begin{align*}
f_1(X_0,X_1,X_2)&=X_0^2+(r+s+1)(X_1-X_2)X_0-X_1X_2, \\
f_2(X_0,X_1,X_2)&=X_1^3X_2^2-X_0X_1(X_1^2+X_2^2)+X_2(X_0^2+X_1^2-X_1X_2) \\
 &\quad +(r+s+1)X_2(X_1-X_2)(X_0-X_1X_2) \\
 &\quad +rsX_1(X_1+X_2)(X_1-X_2)^2.
\end{align*}
By Lemma~\ref{lem:0918}, we have $f_1(w_0,w_1,w_2)=0$ and $f_2(w_0,w_1,w_2)=0$.

Consider the polynomial ring 
\[
\mathscr{P}=\mathbb{C}[Y_\infty,\alpha_0,\alpha_1,\alpha_2,\beta_0,\beta_1,\beta_2].
\]
Let $\chi$ be the homomorphism from $\tilde{\mathscr{R}}$ to $\mathscr{P}$ defined by
$\chi(X_\infty)=Y_\infty$ and 
$\chi(X_j)=\alpha_j+\beta_ji$ for $j=0,1,2$.
Let $\mathcal{J}$ denote the ideal of the polynomial ring $\mathscr{P}$  generated by 
$\chi(\tilde{\mathcal{I}})$, $\chi(f_1(X_0,X_1,X_2))$, $\chi(f_2(X_0,X_1,X_2))$
and $\alpha_j^2+\beta_j^2-1$ for $j=0,1,2$.
With the help of Magma,
we can verify that $\mathcal{J}$ contains
\begin{align*}
& (L(k)-M(k))^2\alpha_1^2+2(L(k)^2-M(k)^2)\alpha_1+(L(k)+M(k))^2-S(k),\\
& 2(k+rs)^2rs\alpha_0-2k^2(k+rs)^2\alpha_1+h_0,\\
& 2(k+rs)^3\alpha_1-2(k+rs)r(r+1)s(s+1)\alpha_2+\ell_0.
\end{align*}
Therefore we have the assertion.
\end{proof}

\begin{lem}\label{lem:0729}
If the real part of $w_1$ 
is in the interval $[-1,1]$,
then the following holds:
\begin{enumerate}
\item\label{0729i} $S(k)\geq0$,
\item\label{0729ii} $M(k)\leq \frac{\sqrt{S(k)}}{2}\leq L(k) \quad\text{or}\quad M(k)\leq \frac{-\sqrt{S(k)}}{2}\leq L(k)$.
\end{enumerate}
\end{lem}
\begin{proof}
By \ref{rEq0_1} in Lemma~\ref{lem:rEq0}
and (\ref{eq:1224-2}) we have $L(k)-M(k)\ne0$.
Assume that $a_1\in[-1,1]$.
Then by (\ref{eq:0115}),
using the notation of (\ref{eq:1224-1}), (\ref{eq:1224-2}), and (\ref{eq:1224-3}),
we have
\[
 a_1=\frac{-L(k)-M(k)\pm\sqrt{S(k)}}{L(k)-M(k)}.
\]
Since $a_1\in\R$, we have \ref{0729i}.
Since $a_1\in[-1,1]$, we have \ref{0729ii}.
\end{proof}

\section{Properties of the polynomials $L(X)$, $M(X)$, and $S(X)$ for the case $r+s=-1$} \label{sec.rs-1}
In this section, 
we suppose that $r+s=-1$,
$2r(r+1)\in\Z$, 
and $(r,s)\ne(0,-1)$. 
We consider properties of the polynomials
(\ref{eq:1224-1}), (\ref{eq:1224-2}), and (\ref{eq:1224-3}).
By (\ref{eq:1224-1}), (\ref{eq:1224-2}), and (\ref{eq:1224-3}) we have
\begin{align}
L(X)&=X^3-2(r^2+r-1)X^2-\frac{8r^2+8r-1}{2}X+\frac{r(r+1)(4r^2+4r-1)}{2}, \label{eq:LL}\\
M(X)&=L(X)-4(X-r(r+1))^2, \label{eq:MM}\\
S(X)&=-(X-r(r+1))S_1(X),\label{eq:SS}
\end{align}
where
\begin{align}
S_1(X)&=4r(r+1)(X-s_+)(X-s_-),\label{eq:S1}\\
s_{\pm}&=2r(r+1)+
\frac{1\pm\sqrt{16r^2(r+1)^2+1}}{8r(r+1)}.\label{eq:Spm}
\end{align}

\begin{lem} \label{lem:0512}
We have $r(r+1)<s_{\pm}$.
\end{lem}
\begin{proof}
Since $s_-<s_+$ by (\ref{eq:Spm}),
we show that $r(r+1)<s_-$.
To do this, we have only to show that
$\sqrt{16r^2(r+1)^2+1}<8r^2(r+1)^2+1$ by (\ref{eq:Spm}).
Since $(8r^2(r+1)^2+1)^2-(16r^2(r+1)^2+1)=64r^4(r+1)^4>0$,
we have the assertion.
\end{proof}

\begin{lem} \label{lem:S1}
Suppose that $r(r+1)<x$.
Then $S(x)\geq0$ if and only if $s_-\leq x\leq s_+$.
\end{lem}
\begin{proof}
This follows easily from (\ref{eq:SS}), (\ref{eq:S1}), and Lemma~\ref{lem:0512}.
\end{proof}

\begin{lem}\label{lem:con}
We have $\Z\cap\{x\mid s_-\leq x\leq s_+\}=\{2r(r+1)\}$.
\end{lem}
\begin{proof}
Note that $r(r+1)\in\Z$ by \eqref{rsZ}.
It is easy to show that $s_-<2r(r+1)<s_+$ by (\ref{eq:Spm}).
Thus, it is enough to show that $2r(r+1)-1<s_-$ and $s_+<2r(r+1)+1$,
or equivalently, $\sqrt{16r^2(r+1)^2+1}<8r(r+1)\pm1$.
We have only to show that $\sqrt{16r^2(r+1)^2+1}<8r(r+1)-1$.
Since
\[
(8r(r+1)-1)^2-(16r^2(r+1)^2+1)=16r(r+1)(3r(r+1)-1)>0,
\]
the result holds.
\end{proof}

\begin{lem}\label{lem:0513-1}
Suppose that $z\in\Z$ and $r(r+1)<z$.
Then $M(z)\leq\frac{\sqrt{S(z)}}{2}\leq L(z)$ or $M(z)\leq\frac{-\sqrt{S(z)}}{2}\leq L(z)$ holds
if and only if $z=2r(r+1)$.
\end{lem}
\begin{proof}
First suppose that $M(z)\leq\frac{\sqrt{S(z)}}{2}\leq L(z)$ or
$M(z)\leq\frac{-\sqrt{S(z)}}{2}\leq L(z)$ holds.
Since $S(z)\geq0$, by Lemma~\ref{lem:S1} we have $s_-\leq z\leq s_+$.
By Lemma~\ref{lem:con}, we have $z=2r(r+1)$.
Secondly suppose that $z=2r(r+1)$.
Since
\begin{align*}
L(2r(r+1))&=\frac{r(r+1)(2r+1)^2}{2}, \\
M(2r(r+1))&=\frac{-r(r+1)(4r(r+1)-1)}{2}<0, \\
\sqrt{ S(2r(r+1)) }&=r(r+1)
\end{align*}
by (\ref{eq:LL}), (\ref{eq:MM}), and (\ref{eq:SS}),
we have $M(2r(r+1))\leq\frac{\sqrt{S(2r(r+1))}}{2}\leq L(2r(r+1))$.
\end{proof}

\begin{lem} \label{lem:0902-4}
Let $W_1$  be the matrix defined by \eqref{eq:W1}, and
$W$ be the matrix defined by \eqref{eq:W}.
Suppose that $W$ is a complex Hadamard matrix.
If $r+s=-1$, then we have \ref{tii} in Theorem~\ref{thm:main}.
\end{lem}
\begin{proof}
By Lemma~\ref{lem:rEq0}, we have $(r,s)\neq(0,-1)$.
Thus, we may use results of this section. In particular,
by Lemmas~\ref{lem:0513-1} and \ref{lem:0729},
we have $k=2r(r+1)$.
By Section~\ref{sec.srg}, $\Gamma$ is a conference graph on $(2r+1)^2$ vertices.

By (\ref{eq:0115}) we have $2r^3(r+1)^3a_1((2r+1)a_1+1)=0$.
Hence $a_1=0$ or 
$a_1=-1/(2r+1)$.
If $a_1=0$ then by (\ref{eq:0804}), (\ref{eq:0805}) we have $a_0=-1$, $a_2=0$, respectively.
By $w_1\ne w_2$ we have $(b_0,b_1,b_2)=(0,\pm1,\mp1)$.
Therefore we have (a) of \ref{tii} in Theorem~\ref{thm:main}.
If $a_1=-1/(2r+1)$ 
then by (\ref{eq:0804}), (\ref{eq:0805}) we have $a_0=1$, 
$a_2=-1/(2r+1)$, 
respectively. By $w_1\ne w_2$ we have 
$(b_0,b_1,b_2)=(0,
\frac{\pm\sqrt{4r^2(r+1)^2-1}}{2r(r+1)},\frac{\mp\sqrt{4r^2(r+1)^2-1}}{2r(r+1)})$.
Therefore we have (b) of \ref{tii} in Theorem~\ref{thm:main}.
\end{proof}

\section{Properties of the polynomials $L(X)$, $M(X)$, and $S(X)$ for the case $r+s\geq0$} \label{sec.6}
In this section, we suppose that $r,s\in\Z$ and $r+s\geq0$.
We further assume $r\geq2$ and $s\leq-2$.
We consider properties of the polynomials
\eqref{eq:1224-1}, \eqref{eq:1224-2}, and \eqref{eq:1224-3}.
We put
\begin{align}
h&=\sqrt{4r(r+1)s(s+1)+1}, \label{hh}\\
\alpha_{\pm}&=\frac{r+s-1}{2}\pm\frac{\sqrt{(s-1)^2-6rs+r(r-2)}}{2}, \label{alpha}\\
\beta_{\pm}&=-rs-\frac{1}{2}\pm\frac{h}{2}, \label{beta}\\
\gamma_{\pm}&=\frac{r+s+3}{2}\pm\frac{\sqrt{r^2+2(5s+3)r+(s+3)^2}}{2}, \label{gamma}\\
\delta&=-rs+\sqrt{r(r+1)s(s+1)}. \label{v}
\end{align}
Then $\alpha_{\pm}$, $\beta_{\pm}$, $\delta\in\R$.
By (\ref{eq:1224-1}), (\ref{eq:1224-2}), and (\ref{eq:1224-3}) we have
\begin{align}
L(X)^2-\frac{S(X)}{4}&=(X-\alpha_-)(X-\alpha_+)(X-\beta_-)^2(X-\beta_+)^2, \label{eq:1208-6} \\
M(X)^2-\frac{S(X)}{4}&=(X-\gamma_-)(X-\gamma_+)(X-(\beta_-+1))^2(X-(\beta_++1))^2. \label{eq:1208-7}
\end{align}

\begin{lem} \label{lem:a-d}
We have the following:
\begin{enumerate}
\item\label{a-di} $\alpha_{\pm},\beta_-+1<-rs$,
\item\label{a-dii} $-rs<\beta_+<\delta<\beta_++1$,
\item\label{a-diii} if $\gamma_{\pm}\in\R$ then $\gamma_{\pm}<-rs$.
\end{enumerate}
\end{lem}
\begin{proof}
\ref{a-di} The inequality $\beta_-+1<-rs$ follows easily from (\ref{beta}).
Since $\alpha_-<\alpha_+$, it remains to show that $\alpha_+<-rs$.
Then by (\ref{alpha}) we have only to  show that
\[
\sqrt{(s-1)^2-6rs+r(r-2)}<-2rs-(r+s-1).
\]
Since $-2rs-(r+s-1)=-(2s+1)r-s+1>0$ and
\begin{align*}
& (-(2s+1)r-s+1)^2-((s-1)^2-6rs+r(r-2)) \\
&=4r(r+1)s(s+1)>0,
\end{align*}
we have $\alpha_+<-rs$.

\ref{a-dii} First, the inequality $-rs<\beta_+$ follows easily from the definition \eqref{beta}.
Since
\[
h<1+2\sqrt{r(r+1)s(s+1)},
\]
we have
$\beta_+<\delta$ from the definitions
(\ref{beta}) and (\ref{v}), while 
$\delta<\beta_++1$ follows trivially from these.

\ref{a-diii} Since $\gamma_-<\gamma_+$,
it is enough to show that $\gamma_+<-rs$.
By (\ref{gamma}) we have only to  show that
\[
\sqrt{r^2+2(5s+3)r+(s+3)^2}<-2rs-(r+s+3).
\]
Since $-2rs-(r+s+3)=-(2s+1)r-(s+3)>0$ and
\begin{align*}
& (-(2s+1)r-(s+3))^2-(r^2+2(5s+3)r+(s+3)^2) \\
&=4r(r+1)s(s+1)>0,
\end{align*}
we have $\gamma_+<-rs$. 
\end{proof}

\begin{lem} \label{lem:LMS-rs}
We have $L(-rs)=M(-rs)<0$.
\end{lem}
\begin{proof}
By (\ref{eq:1224-1}) and (\ref{eq:1224-2}) we have
\[
L(-rs)=M(-rs)=\frac{r(r+1)s(s+1)((2r+1)(s+1)-r)}{2}<0.
\]
\end{proof}

\begin{lem} \label{lem:LL}
We have the following:
\begin{enumerate}
\item\label{LL1}  $L(X)$ has exactly one real root $\zeta$ in $(-rs,\infty)$,
and $\beta_+\leq\zeta<\delta$,
\item\label{LL2}  $L(x)<0$ for $-rs<x<\zeta$, and 
$L(x)\geq0$ for $\zeta\leq x$.
\end{enumerate}
\end{lem}
\begin{proof}
Since $L'(X)=3(X-\theta_-)(X-\theta_+)$, where
\begin{align*}
\theta_{\pm}&=\frac{(-4s+1)r+s-3}{6}\pm\frac{\sqrt{\iota}}{6}, \\
\iota&=(16s(s+1)+1)r^2+2(8s^2+s-3)r+s^2-6s+3>0,
\end{align*}
$L(X)$ has the local maximum at $X=\theta_-$ and the local minimum at $X=\theta_+$.

\ref{LL1} We show that (a) $\theta_-<-rs<\theta_+$, 
(b) $\theta_+<\beta_+$, $L(\beta_+)\leq0$, and $L(\delta)>0$.
Then (a) together with Lemma~\ref{lem:LMS-rs} implies that the first half of 
\ref{LL1} holds, and
(b) implies that the latter half of \ref{LL1} holds.

First we show that (a) holds.
Since
\begin{align*}
6(-rs-\theta_-)&>-6rs-((-4s+1)r+s-3) \\
&=-(2s+1)r-s+3>0,
\end{align*}
we have $\theta_-<-rs$.
To show that $-rs<\theta_+$,
we have only to show that $-(2s+1)r-s+3<\sqrt{\iota}$.
Since
\[
\iota-(-(2s+1)r-s+3)^2=12r(r+1)s(s+1)-6>0,
\]
we have $-rs<\theta_+$.
Hence $\theta_-<-rs<\theta_+$.

Secondly we show that (b) holds.
To show that $\theta_+<\beta_+$,
by (\ref{beta}) and the definition of $\theta_+$
we have only to show that
$\sqrt{\iota}<-(2s+1)r-s+3h$.
Since
\begin{align*}
&(-(2s+1)r-s+3h)^2-\iota 
\\&=-6((2s+1)r+s)h 
+24s(s+1)r^2+6(4s^2+6s+1)r+6(s+1)
\\&>0,
\end{align*}
we have $\theta_+<\beta_+$.
We have
\[
L(\beta_+)=\frac{\kappa_1h+\kappa_2}{4},
\]
where
\begin{align*}
\kappa_1&=(2s+1)r+s<0, \\
\kappa_2&=4s(s+1)r^2+(2s(2s+1)-1)r-s.
\end{align*}
Since
\[
\kappa_1^2h^2-\kappa_2^2=4r(r+1)s(s+1)(r+s)(r+s+2)\geq0,
\]
by our assumption, we have
\begin{equation}\label{eq:0425-1}
L(\beta_+)\leq0.
\end{equation}
We have
\begin{equation}\label{eq:0425-2}
L(\delta)=\frac{\sqrt{r(r+1)s(s+1)}}{2}+2r(r+1)s(s+1)>0.
\end{equation}

\ref{LL2} This follows easily from \ref{LL1}, (\ref{eq:0425-1}), and (\ref{eq:0425-2}).
\end{proof}

\begin{lem}\label{lem:MM}
We have the following:
\begin{enumerate}
\item\label{MM1} $M(X)$ has exactly one real root $\eta$ in $(-rs,\infty)$,
and $\delta<\eta\leq\beta_++1$,
\item\label{MM2} $M(x)\leq0$ for $-rs<x\leq\eta$, and 
$M(x)>0$ for $\eta<x$.
\end{enumerate}
\end{lem}
\begin{proof}
Since $M'(X)=3(X-\nu_-)(X-\nu_+)$, where
\begin{align*}
\nu_{\pm}&=\frac{(-4s+1)r+s+5}{6}\pm\frac{\sqrt{\rho}}{6}, \\
\rho&=(16s(s+1)+1)r^2+2(8s^2+17s+5)r+s^2+10s+19>0,
\end{align*}
$M(X)$ has the local maximum at $X=\nu_-$ and the local minimum at $X=\nu_+$.

\ref{MM1} It is enough to show that (a) $\nu_-<-rs<\nu_+$, 
(b) $\nu_+<\delta$, $M(\delta)<0$, and $M(\beta_++1)\geq0$.
Then (a) together with Lemma~\ref{lem:LMS-rs} implies that 
the first half of \ref{MM1} holds, and
(b) implies that the latter half of \ref{MM1} holds.

First we show that (a) holds.
Since
\begin{align*}
6(-rs-\nu_-)&=-(2s+1)r-s-5+\sqrt{\rho},\\
&>-(2s+1)r-s-5>0,
\end{align*}
we have $\nu_-<-rs$.
To show that $-rs<\nu_+$,
we have only to show that $-(2s+1)r-s-5<\sqrt{\rho}$.
Since
\[
\rho-(-(2s+1)r-s+3)^2=12r(r+1)s(s+1)-6>0,
\]
we have $-rs<\nu_+$.
Hence $\nu_-<-rs<\nu_+$.

Secondly we show that (b) holds.
To show that $\nu_+<\delta$,
by (\ref{v}) and the definition of $\nu_+$
we have only to show that
$(2s+1)r+s+5+\sqrt{\rho}<6\sqrt{r(r+1)s(s+1)}$.
Since
\begin{align*}
&(\sqrt{r(r+1)s(s+1)})^2-((2s+1)r+s+5+\sqrt{\rho})^2 \\
&=-((4s+2)r+2s+10)\sqrt{\rho} \\
&\quad +(16s^2+16s-2)r^2+(16s^2-20s-20)r-2s^2-20s-44>0,
\end{align*}
we have $\nu_+<\delta$.
It is easy to show that
\begin{equation}\label{eq:0425-3}
M(\delta)=\frac{\sqrt{r(r+1)s(s+1)}}{2}-2r(r+1)s(s+1)<0.
\end{equation}
We have
\[
M(\beta_++1)=\frac{\sigma_1h+\sigma_2}{4},
\]
where
\begin{align*}
\sigma_1&=-((2s+1)r+s+2)>0,\\
\sigma_2&=-4s(s+1)r^2-(4s^2+6s+1)r-s-2.
\end{align*}
Since
\[
\sigma_1^2h^2-\sigma_2^2=4r(r+1)s(s+1)(r+s)(r+s+2)\geq0,
\]
by our assumption, we have
\begin{equation}\label{eq:0425-4}
M(\beta_++1)\geq0.
\end{equation}

\ref{MM2} This follows easily from \ref{MM1}, (\ref{eq:0425-3}), and (\ref{eq:0425-4}).
\end{proof}

\begin{lem} \label{lem:new-lem}
For $-rs\leq x$ we have the following:
\begin{enumerate}
\item\label{new-lem1} $L(x)^2\geq\frac{S(x)}{4}$,
and equality holds if and only if $x=\beta_+$,
\item\label{new-lem2} $M(x)^2\geq\frac{S(x)}{4}$,
and equality holds if and only if $x=\beta_++1$.
\end{enumerate}
\end{lem}
\begin{proof}
\ref{new-lem1} Since $\alpha_{\pm}<-rs$ by 
\ref{a-di} in Lemma~\ref{lem:a-d},
by \eqref{eq:1208-6} we have the claimed inequality.
Since $\alpha_{\pm},\beta_-<-rs<\beta_+$ by 
\ref{a-di} and \ref{a-dii} in Lemma~\ref{lem:a-d},
equality holds if and only if $x=\beta_+$.

\ref{new-lem2} First suppose that $\gamma_{\pm}\in\R$.
Since $\gamma_{\pm}<-rs$ by \ref{a-diii} in Lemma~\ref{lem:a-d},
by \eqref{eq:1208-7} we have the claimed inequality.
Since $\beta_-+1,\gamma_{\pm}<-rs$ by \ref{a-di} and 
\ref{a-diii} in Lemma~\ref{lem:a-d},
equality holds if and only if $x=\beta_++1$.

Secondly suppose that $\gamma_{\pm}\not\in\R$.
Then $\overline{\gamma_+}=\gamma_-$ by \eqref{gamma},
so we also have the claimed inequality.
Since $\beta_-+1<-rs$ by \ref{a-di} in Lemma~\ref{lem:a-d},
equality holds if and only if $x=\beta_++1$.
\end{proof}

For the remainder of this subsection,
we suppose that $r+s=0$.
By (\ref{eq:1224-1}), (\ref{eq:1224-2}), and (\ref{eq:1224-3}) we have
\begin{align}
L(X)&=\frac{(X-\tau_-)(X-\tau_+)(X-\beta_+)}{2}, \label{eq:LLL}\\
M(X)&=\frac{(2X^2-5X+2r^2+1)(X-(\beta_++1))}{2}, \label{eq:MMM}\\
S(X)&=(X-\beta_+)^2(X-(\beta_++1))^2\geq0, \label{eq:SSS}
\end{align}
where
\begin{align}
\tau_{\pm}&=\frac{-1\pm\sqrt{16r^2+1}}{4}, \nonumber\\
\beta_+&=2r^2-1\in\Z.\label{eq:b+}
\end{align}
By (\ref{alpha}) and (\ref{beta}) we have
\begin{align*}
\alpha_{\pm}&=\frac{-1\pm\sqrt{8r^2+1}}{2},\\
\beta_-&=0.
\end{align*}

\begin{lem}\label{lem:F}
Suppose that $r+s=0$.
Let $\zeta$ and $\eta$ be as defined in {\rm Lemmas~\ref{lem:LL} and \ref{lem:MM}}.
Then we have $\zeta=\beta_+$ and $\eta=\beta_++1$.
\end{lem}
\begin{proof}
We have $\tau_{\pm}<r^2$.
Then by (i) in Lemma~\ref{lem:LL} and (\ref{eq:LLL}) we have $\zeta=\beta_+$.
Since the discriminant of $2x^2-5x+2r^2+1$ is $-16r^2+17<0$,
by (i) in Lemma~\ref{lem:MM} and (\ref{eq:MMM}) we have $\eta=\beta_++1$.
\end{proof}

\begin{lem}\label{lem:0729-1}
Suppose that $r+s=0$, $z\in\Z$ and $r^2<z$.
Then the following are equivalent:
\begin{enumerate}
\item\label{0729-1-1} $S(z)\geq0$ and $M(z)\leq\frac{\sqrt{S(z)}}{2}\leq L(z)$,
\item\label{0729-1-2} $S(z)\geq0$ and $M(z)\leq\frac{-\sqrt{S(z)}}{2}\leq L(z)$,
\item\label{0729-1-3} $S(z)=0$,
\item\label{0729-1-4} $z\in\{\beta_+, \beta_++1\}$.
\end{enumerate}
\end{lem}
\begin{proof}
Since $z\in\Z$ 
and $\beta_+, \beta_++1\in\Z$ by \eqref{eq:b+},
the condition \ref{0729-1-4} is equivalent to
$\beta_+\leq z\leq\beta_++1$.

First suppose that \ref{0729-1-1} holds.
Since $L(z)\geq0$,
by \ref{LL2} in Lemma~\ref{lem:LL} and Lemma~\ref{lem:F} we have
\begin{equation}\label{eq:1929}
\beta_+\leq z.
\end{equation}
Suppose that $M(z)\leq0$.
By \ref{MM2} in Lemma~\ref{lem:MM} and Lemma~\ref{lem:F} 
we have $z\leq\beta_++1$.
By \eqref{eq:1929} we have $z=\beta_+, \beta_++1$.
Suppose that $M(z)>0$.
Then $M(z)^2\leq\frac{S(z)}{4}$.
By \ref{new-lem2} in Lemma~\ref{lem:new-lem} we have $z=\beta_++1$.
Thus we have (iv).

Secondly suppose that \ref{0729-1-2} holds.
Since $M(z)\leq0$,
by \ref{MM2} in Lemma~\ref{lem:MM} and Lemma~\ref{lem:F} we have
\begin{equation}\label{eq:1940}
z\leq\beta_++1.
\end{equation}
Suppose that $L(z)\geq0$.
Then by \ref{LL2} Lemma~\ref{lem:LL} and Lemma~\ref{lem:F} 
we have $\beta_+\leq z$.
By \eqref{eq:1940} we have $z=\beta_+, \beta_++1$.
Suppose that $L(z)<0$.
Then $L(z)^2\leq\frac{S(z)}{4}$.
By \ref{new-lem1} in Lemma~\ref{lem:new-lem} we have $z=\beta_+$.
Thus we have \ref{0729-1-4}.

The equivalence of \ref{0729-1-3} and \ref{0729-1-4} 
follows immediately from \eqref{eq:SSS}.

Finally suppose that \ref{0729-1-4} holds.
Since $S(z)=0$ by \ref{0729-1-3}, 
it suffices to show $L(z)\geq0$ and $M(z)\leq0$.
By \ref{LL2} in Lemma~\ref{lem:LL} and Lemma~\ref{lem:F} we have $L(z)\geq0$.
By \ref{MM2} in Lemma~\ref{lem:MM} and Lemma~\ref{lem:F} we have $M(z)\leq0$.
\end{proof}

\begin{lem}\label{lem:25}
Let $W_1$  be the matrix defined by \eqref{eq:W1}, and
$W$ be the matrix defined by \eqref{eq:W}.
Suppose that $W$ is a complex Hadamard matrix.
If $r+s=0$, then we have \ref{ti} in Theorem~\ref{thm:main}.
\end{lem}
\begin{proof}
By Lemma~\ref{lem:rEq0}, we have $r^2=-rs<k$.
Also, by Lemma~\ref{lem:rEq0}, we have $s\leq-2$ and $r\geq2$.
Thus, we may use results of this section. In particular,
by Lemmas~\ref{lem:0729-1}, \ref{lem:0729}, and \eqref{eq:b+} 
we have $k=2r^2$ or $k=2r^2-1$.
First suppose $k=2r^2$.
By \eqref{eq:0115} we have $a_1=-1$.
Then by \eqref{eq:0804}, \eqref{eq:0805} we have $a_0=1$, $a_2=1$, respectively.
Therefore we have \ref{ti} in Theorem~\ref{thm:main}.
Secondly suppose $k=2r^2-1$.
By (\ref{m1}) we have $m_1=\frac{(2r-1)(2r^2-1)}{2r}$.
This is a contradiction since $m_1$ must be an integer.
\end{proof}

\section{The case $r+s>0$}\label{sec:rs-p}
In this section, we suppose that $r,s\in\Z$ and $r+s>0$.
Then by Lemma~\ref{lem:rEq0} we have $r\geq3$ and $s\leq-2$.
We consider properties of the polynomials
(\ref{eq:1224-1}), (\ref{eq:1224-2}), and (\ref{eq:1224-3}).
Let
\begin{align*}
\kappa(x)&=(2s+1)^2x^3 -(2s+1)(8s^3-2s^2-s+2)x^2 \nonumber\\
&\quad -(16s^5+8s^2+2s-1)x +4s^2 + s.
\end{align*}

\begin{lem}\label{lem:kappa}
Assume that $-s+1\leq x<-2s+1$.
Then we have $\kappa(x)<0$.
\end{lem}
\begin{proof}
Since 
\begin{align*}
\kappa(-s)&=-4s^2(s-1)(2s(s+1)+1)>0,\\
\kappa(-s+1)&=s^2(8s^3+4s^2-8s+1)<0,\\
\kappa(-2s+1)&=-32s^4(s^2-1)<0,
\end{align*}
we have the assertion.
\end{proof}

Let
\begin{align}
\psi(x)&=(s+1)(x+1)((2s+1)x-1), \label{eq:0430} \\
\phi(x)&=2\psi(x)-(2s+1)(x+2s-1).\label{eq:0513-1}
\end{align}

\begin{lem}\label{lem:0430}
We have $\psi(r)>0$ and $\phi(r)>0$.
\end{lem}
\begin{proof}
The inequality $\psi(r)>0$ follows immediately by (\ref{eq:0430}).
Since $r>-s$ and
\begin{align*}
\phi(0)&=-2s(2s+1)-1<0,\\
\phi(-s)&=(s-1)\left(4s^2(s+2)-2(s+1)(s-2)-3\right)>0,
\end{align*}
we have $\phi(r)>0$.
\end{proof}

Let $h$ be defined as (\ref{hh}).

\begin{lem}\label{lem:bound1}
Assume that $k=-rs+\frac{h+\epsilon}{2}$ and $h\in\Z$, where $\epsilon\in\{\pm1\}$.
Then we have
\[
n\geq -(2s+1)r+2+\frac{2\psi(r)}{h+1}.
\]
\end{lem}
\begin{proof}
First we show that
\begin{equation}\label{eq:h-bound}
h\geq-2(s+1)r+3.
\end{equation}
To do this, since $h>0$ and
\[
h^2-(-2(s+1)r+1)^2=4r(s+1)(s-r+1)>0,
\]
we have $h>-2(s+1)r+1$. Since $h$ is odd,
we have (\ref{eq:h-bound}).

Secondly we show the assertion.
Since 
\begin{align}
k&\geq-rs+\frac{h-1}{2} \nonumber\\
&\geq-(2s+1)r+1, && \text{(by (\ref{eq:h-bound}))} \label{eq:0515} 
\end{align}
we have
\begin{align*}
n&=1+k+\ell\\
&=1+k-\frac{k(r+1)(s+1)}{k+rs} && \text{(by (\ref{mu0}))} \\
&\geq1-(2s+1)r+1+\frac{(s+1)(r+1)((2s+1)r-1)}{k+rs} && \text{(by (\ref{eq:0515}))} \\
&\geq -(2s+1)r+2+\frac{2\psi(r)}{h+1} && \text{(by  (\ref{eq:0430}))}.\qedhere
\end{align*}
\end{proof}

Let $u=r+s$. Then $u\in\Z$ and
\begin{equation}\label{eq:20191101}
1\leq u\leq r-2.
\end{equation}

\begin{lem} \label{lem:0902-1}
The polynomial $S''(X)$ has two distinct real roots:
\begin{equation}\label{eq:0902-1}
\tau_{\pm}=\frac{c_1\pm\sqrt{c_2}}{6(u+1)^2},
\end{equation}
where
\begin{align*}
c_1&=3(u+1)^2(2r(r-u)-1)+3(r(r+1)+(r-u)(r-u-1)), \\
c_2&=12r(r+1)u(u+2)(u^2+2u-2)(r-u)(r-u-1)+3(u+1)^2.
\end{align*}
\end{lem}
\begin{proof}
Observe $c_2>0$ follows from \eqref{eq:20191101}.
Since
\begin{align*}
S''(X)&=12(u+1)^2X^2 \\
&\quad -12\left(2(u^2+2u+2)r^2-2u(u^2+2u+2)r-(u+1)\right)X \\
&\quad +8(u^2+2u+6)r^4-16u(u^2+2u+6)r^3 \\
&\quad +4(2u^4+5u^3+15u^2-4u-6)r^2 \\
&\quad -4u(u+1)(u^2+2u-6)r+2
\end{align*}
by (\ref{eq:1224-3}), we have (\ref{eq:0902-1}).
\end{proof}

\begin{lem} \label{lem:0902-2}
Let $\tau_{\pm}$ be the real number defined by {\rm(\ref{eq:0902-1})}.
Then $\tau_{\pm}<\beta_+$.
\end{lem}
\begin{proof}
Since $\tau_-<\tau_+$, it is enough to show that $\tau_+<\beta_+$.
By (\ref{beta}), we have
\[
\beta_+=r(r-u)-\frac12+\frac{h}{2}.
\]
Since
\[
h^2-(2r(r-u-1)+1)^2=4r(2r-u-1)(r-u-1)>0,
\]
we have 
\begin{equation}\label{eq:beta1101a}
\beta_+-r(2(r-u)-1)=\frac{1}{2}(h-(2r(r-u-1)+1))>0.
\end{equation}
Since
\begin{align*}
& \left( 6(u+1)^2r( 2(r-u)-1 )-c_1 \right)^2-c_2\\
&=6(u+1)^2\left(2r(r-u-1)\left( u(u+2)( 2r(r-u-2)+u)+3\right)+1\right)>0,
\end{align*}
we have 
\begin{equation}\label{eq:beta1101b}
r(2(r-u)-1)-\tau_+=\frac{6(u+1)^2r( 2(r-u)-1)-c_1-\sqrt{c_2}}{6(u+1)^2}>0.
\end{equation}
By (\ref{eq:beta1101a}) and (\ref{eq:beta1101b}), we obtain $\tau_+<\beta_+$.
\end{proof}

Define
\begin{align*}
g_1&=4u(u+2)(u(u+2)-2)r(r+1)(r-u)(r-u-1)+(u+1)^2, \\
g_2&=2r(r+1)(r-u)(r-u-1)\nonumber\\
&\quad \times\left(8u(u+2)r(r+1)(r-u)(r-u-1)+7u(u+2)-1\right)-1,\\
g_3&=16u(u+2)r(r+1)(r-u)(r-u-1)-1.
\end{align*}

\begin{lem}\label{lem:g123}
We have $g_1>0$, $g_2>0$, and $g_3>0$.
\end{lem}
\begin{proof}
These follow immediately from (\ref{eq:20191101}).
\end{proof}

\begin{lem} \label{lem:SS}
The polynomial $S(X)$ has exactly two real roots, say, $\xi_1, \xi_2$, and
$\beta_+<\xi_1<\delta<\xi_2<\beta_++1$. Moreover, 
both $\xi_1$ and $\xi_2$ are simple.
\end{lem}
\begin{proof}
Set $f_0(X)=S(X)$ and $f_1(X)=f_0'(X)$.
Set
\[
f_j(X)=-\text{Rem}(f_{j-2}(X),f_{j-1}(X))
\]
for $j=2,3,4$.
Let $c_j$ be the leading coefficient of $f_j(X)$, and $d_j=\deg f_j(X)$.
We have $(d_0,d_1,d_2,d_3,d_4)=(4,3,2,1,0)$.
Then we have the following:
\begin{align*}
c_0&=(u+1)^2>0, \\
c_1&=4(u+1)^2>0, \\
c_2&=\frac{g_1}{4(u+1)^2}, \\
c_3&=\frac{-32u^2(u+1)^2(u+2)^2r(r+1)(r-u)(r-u-1)g_2}{g_1^2}, \\
c_4&=\frac{-r^2(r+1)^2(r-u)^2(r-u-1)^2g_1^2g_3}{4(u+1)^2g_2^2}.
\end{align*}
By Lemma~\ref{lem:g123} we have $c_2>0$, $c_3<0$, and $c_4<0$.
Therefore we have Table~\ref{tbl:strum}.
\begin{table}[b]
\begin{center}
\caption{Sturm's sequence}
\begin{tabular}{|c||c|c|c|c|c||c|}
\hline
$j$ & 0 & 1 & 2 & 3 & 4 & $\sharp$ sign changes \\ \hline
$\text{sgn}(c_j)$ & $+$ & $+$ & $+$ & $-$ & $-$ & $1$ \\ \hline
$\text{sgn}((-1)^{d_j}c_j)$ & $+$ & $-$ & $+$ & $+$ & $-$ & $3$ \\ 
\hline
\end{tabular}
\label{tbl:strum}
\end{center}
\end{table}
Applying Theorem~\ref{thm:sturm} for $S(X)$ using Table~\ref{tbl:strum},
we see that $S(X)$ has exactly two real roots.

We show that $S(\beta_+)>0$, $S(\beta_++1)>0$, and $S(\delta)<0$.
We have
\[
S(\beta_+)=\frac{h_1h+h_2}{2},
\]
where
\begin{align*}
h_1&=-(2r(r-u)-u)\left((4(r-u)r-2(2u+1))(r-u)r-u\right), \\
h_2&=u^2+2r(r-u) \\
& \quad\times \left( 8r^2(r-u)^2(r(r-u)-(2u+1))+u^2(9r(r-u)-u+1)\right. \\
&\qquad\quad \left.+2r(r-u)(3u+1)\right)>0
\end{align*}
since $r-u\geq2$.
Since
\[
h_2^2-h_1^2h^2=4r^2(r+1)^2u^2(u+2)^2(r-u)^2(r-u-1)^2>0,
\]
we have $S(\beta_+)>0$.
We have
\[
S(\beta_++1)=\frac{h_3h+h_4}{2},
\]
where
\begin{align*}
h_3&=-(2r(r-u)-(u+2))\left( (4r(r-u)-2(2u+3))(r-u)r+u+2\right), \\
h_4&=u^2+2(r+1)(r-u-1) \\
     &\quad\times \left(r(r-u)(8((r-u)r-(u+2))(r-u)r+u^2+6u+10)-2\right)>0
\end{align*}
since $r-u\geq2$. Since
\[
h_4^2-h_3^2h^2=4r^2(r+1)^2u^2(u+2)^2(r-u)^2(r-u-1)^2>0,
\]
we have $S(\beta_++1)>0$.
We have
\[
S(\delta)=r(r+1)(r-u)(r-u-1)(h_5\sqrt{r(r+1)(r-u)(r-u-1)}+h_6),
\]
where
\begin{align*}
h_5&=-4(2r(r-u)-(u+1))<0, \\
h_6&=8r(r+1)(r-u)(r-u-1)+1.
\end{align*}
Since
\[
h_5^2r(r+1)(r-u)(r-u-1)-h_6^2=r(r+1)u(u+2)(r-u)(r-u-1)-1>0,
\]
we have
\begin{equation} \label{eq:delta0220}
S(\delta)<0.
\end{equation}
The polynomial $S(X)$ has exactly two real roots, say, $\xi_1, \xi_2$, and
$\beta_+<\xi_1<\delta<\xi_2<\beta_++1$.

We show that the roots $\xi_1, \xi_2$ are simple.
Since $\deg S(X)=4$ and the number of imaginary roots of $S(X)$ is even,
the sum of multiplicities of $\xi_1$ and $\xi_2$ is $2$ or $4$.
If both $\xi_1$ and $\xi_2$ are double roots, then 
$S(x)>0$ for $\xi_1<x<\xi_2$. 
This contradicts (\ref{eq:delta0220}).
By Lemmas~\ref{lem:0902-1},~\ref{lem:0902-2} we have $S''(x)\ne0$ for $\beta_+\leq x\leq\beta_++1$.
Thus neither $\xi_1$ nor $\xi_2$ is triple.
\end{proof}

\begin{lem} \label{lem:0418}
We have $L(\beta_+)\leq0$ and $M(\beta_++1)\geq0$.
\end{lem}
\begin{proof}
We have
\[
L(\beta_+)=\frac{ \tau_1h+\tau_2}{4},
\]
where
\begin{align*}
\tau_1&=-2r(r-u)+u<0, \\
\tau_2&=4r^4-8r^3u+(4u^2-4u-2)r^2+2u(2u+1)r-u.
\end{align*}
Since
\[
\tau_1^2h^2-\tau_2^2=4r(r+1)u(u+2)(r-u)(r-u-1)\geq0,
\]
we have $L(\beta_+)\leq0$.
Also, we have
\[
M(\beta_++1)=\frac{ \tau_3h+\tau_4}{4},
\]
where
\begin{align*}
\tau_3&=2r(r-u)-u-2>0, \\
\tau_4&=-4r^4+8ur^3-(4u^2-4u-6)r^2-2u(2u+3)r-u-2.
\end{align*}
Since
\begin{align*}
\tau_3^2h^2-\tau_4^2=4r(r+1)u(u+2)(r-u)(r-u-1)\geq0,
\end{align*}
we have $M(\beta_++1)\geq0$.
\end{proof}

\begin{lem} \label{lem:D}
We have $\xi_1<\zeta<\eta<\xi_2$.
\end{lem}
\begin{proof}
Suppose that $\zeta\leq\xi_1$.
By \ref{LL1} in Lemma~\ref{lem:LL} we have $L(\zeta)=0$.
By \ref{new-lem1} in Lemma~\ref{lem:new-lem} we have $L(\zeta)^2\geq\frac{S(\zeta)}{4}$,
and by Lemma~\ref{lem:SS} we have $\frac{S(\zeta)}{4}\geq0$.
Hence $L(\zeta)^2=\frac{S(\zeta)}{4}=0$.
This contradicts (\ref{eq:1208-6}) and Lemma~\ref{lem:0418}.

Suppose that $\xi_2\leq\eta$.
By \ref{MM1} in Lemma~\ref{lem:MM} we have $M(\eta)=0$.
By \ref{new-lem2} in Lemma~\ref{lem:new-lem} we have $M(\eta)^2\geq\frac{S(\eta)}{4}$,
and by Lemma~\ref{lem:SS} we have $\frac{S(\eta)}{4}\geq0$.
Hence $M(\eta)^2=\frac{S(\eta)}{4}=0$.
This contradicts \eqref{eq:1208-7} and Lemma~\ref{lem:0418}.

We have 
\begin{equation}\label{eq:0612}
M(x)\leq L(x)
\end{equation}
for $x\in\R$ by \eqref{eq:1224-2}.
The inequality $\zeta<\eta$ follows from (\ref{eq:0612}), 
\ref{LL1} in Lemma~\ref{lem:LL}, and \ref{MM1} in Lemma~\ref{lem:MM}.
\end{proof}

Let
\begin{align}
A&=(-rs,\xi_1], \label{int-A}\\
B&=[\xi_2,\infty).\label{int-B}
\end{align}

\begin{lem} \label{lem:B}
We have the following:
\begin{enumerate}
\item\label{B1}  $S(x)\geq0$ for $x\in\R$ holds if and only if $x\in A\cup B$.
\item\label{B2}  For $x\in A\cup B$,
\begin{itemize}
\item[\rm(a)]  $M(x)\leq\frac{\sqrt{S(x)}}{2}\leq L(x)$ holds
if and only if $x=\beta_++1$,
\item[\rm(b)]   $M(x)\leq\frac{-\sqrt{S(x)}}{2}\leq L(x)$ holds
if and only if $x=\beta_+$.
\end{itemize}
\end{enumerate}
\end{lem}
\begin{proof}
\ref{B1} This follows from Lemma~\ref{lem:SS} since the leading coefficient of $S(X)$ is positive.

\ref{B2} (a) Suppose that $M(x)\leq\frac{\sqrt{S(x)}}{2}\leq L(x)$ holds.
Since $L(x)\geq0$,
by \ref{LL2} in Lemma~\ref{lem:LL} we have $\zeta\leq x$.
Since $x\in A\cup B$, by Lemma~\ref{lem:D} we have $x\in B$.
Hence $\eta<x$.
Then by \ref{MM2} in Lemma~\ref{lem:MM} we have $M(x)\geq0$.
Hence $M(x)^2\leq\frac{S(x)}{4}$.
By \ref{new-lem2} in Lemma~\ref{lem:new-lem} we have $x=\beta_++1$.

Conversely, suppose that $x=\beta_++1$.
By \ref{new-lem2} in Lemma~\ref{lem:new-lem} and Lemma~\ref{lem:0418} 
we have $M(\beta_++1)=\frac{\sqrt{S(\beta_++1)}}{2}$.
Since $-rs<\beta_++1$ by \ref{a-dii} in Lemma~\ref{lem:a-d},
by \eqref{eq:1224-2} we have $M(\beta_++1)<L(\beta_++1)$.
Therefore $M(\beta_++1)=\frac{\sqrt{S(\beta_++1)}}{2}<L(\beta_++1)$.

\ref{B2} (b) Suppose that $M(x)\leq\frac{-\sqrt{S(x)}}{2}\leq L(x)$ holds.
Since $M(x)\leq0$,
by \ref{MM2} in Lemma~\ref{lem:MM} we have $x\leq\eta$.
Since $x\in A\cup B$, by Lemma~\ref{lem:D} we have $x\in A$.
Hence $x<\zeta$.
Then by \ref{LL2} in Lemma~\ref{lem:LL} we have $L(x)<0$.
Thus $L(x)^2\leq\frac{S(x)}{4}$.
By \ref{new-lem1} in Lemma~\ref{lem:new-lem} we have $x=\beta_+$.

Conversely, suppose that $x=\beta_+$.
By \ref{new-lem1} in Lemma~\ref{lem:new-lem} and Lemma~\ref{lem:0418}
we have $\frac{-\sqrt{S(\beta_+)}}{2}=L(\beta_+)$.
Since $-rs<\beta_+$ by \ref{a-dii} in Lemma~\ref{lem:a-d},
by \eqref{eq:1224-2} we have $M(\beta_++1)<L(\beta_++1)$.
Therefore $M(\beta_+)<\frac{-\sqrt{S(\beta_+)}}{2}=L(\beta_+)$.
\end{proof}

For the remainder of this section,
we assume that $W$ defined by \eqref{eq:W} is a complex Hadamard matrix for the case $r+s>0$.
By \ref{B1} in Lemma~\ref{lem:B} and \ref{0729i} in Lemma~\ref{lem:0729} we have 
$k\in A\cup B$ by \eqref{int-A} and \eqref{int-B}.
By \ref{B2} (a) and (b) in Lemma~\ref{lem:B} and \ref{0729ii} in Lemma~\ref{lem:0729} we have 
$k\in\{\beta_+,\beta_++1\}$, that is, $k=-rs+\frac{h+\epsilon}{2}$, where $\epsilon\in\{\pm1\}$.
Then by (\ref{beta}) we have $h\in\Z$.
By \ref{rEq0_2} in Lemma~\ref{lem:rEq0} and Lemma~\ref{lem:bound1} we have
\begin{equation}\label{eq:0513}
4s^2-1\geq-(2s+1)r+2+\frac{2\psi(r)}{h+1}.
\end{equation}
Since
\begin{align*}
0&<\frac{2\psi(r)}{h+1} && \text{(by Lemma~\ref{lem:0430})} \\
 &\leq(2s+1)(r+2s-1)-2 && \text{(by (\ref{eq:0513}))} \\
 &<(2s+1)(r+2s-1),
\end{align*}
we have $r<-2s+1$.
Then by Lemma~\ref{lem:kappa} we have $\kappa(r)<0$.

By (\ref{eq:0513}) we have
\begin{align*}
(2s+1)(r+2s-1)h &>2\psi(r)-(2s+1)(r+2s-1) \\
&=\phi(r) && \text{(by (\ref{eq:0513-1}))}\\
&>0 && \text{(by Lemma~\ref{lem:0430})}.
\end{align*}
Since
\begin{align*}
0&<((2s+1)(r+2s-1)h)^2-\phi(r)^2 \\
&=-4(s+1)(r+1)\kappa(r),
\end{align*}
we have $\kappa(r)>0$.
This is a contradiction.
Therefore there does not exist such a complex Hadamard matrix.

\subsection*{Acknowledgements}
The authors are grateful to the anonymous reviewers 
whose suggestions improved the presentation. In particular,
one of the reviewer pointed out an earlier result
\cite[Proposition 3.4.16]{FS} (Remark~\ref{rem:2}), 
suggested to consider
equivalence (Remark~\ref{rem:3}), and proposed to use the extra
indeterminate $X_\infty$.


\end{document}